\documentclass{amsart}
\usepackage{graphicx}
\theoremstyle{plain}
\newtheorem{theorem}{Theorem}[section]

\newtheorem{corollary}[theorem]{Corollary}
\newtheorem{proposition}[theorem]{Proposition}

\theoremstyle{definition}
\newtheorem{definition}[theorem]{Definition}

\theoremstyle{remark}
\newtheorem{remark}{Remark}

\newcommand{\C}{\mathbb{C}}
\newcommand{\B}{\mathbb{B}}
\newcommand{\h}{\mathcal{H}}

\begin{document}

\title[An extension of Birkhoff--James orthogonality relations]
{An extension of Birkhoff--James orthogonality relations in semi-Hilbertian space operators}

\author{S. Mojtaba Enderami}
\address{School of Mathematics and Computer Sciences,
Damghan University, Damghan, P.O.BOX 36715-364, Iran}
\email{sm.enderami@std.du.ac.ir}

\author{Mortaza Abtahi${}^*$}
\address{School of Mathematics and Computer Sciences,
Damghan University, Damghan, P.O.BOX 36715-364, Iran}
\email{abtahi@du.ac.ir}

\author{Ali Zamani}
\address{School of Mathematics and Computer Sciences,
Damghan University, Damghan, P.O.BOX 36715-364, Iran}
\email{zamani.ali85@yahoo.com}

\thanks{${}^*$ Corresponding author; Tel: +982335220092,
Email: \texttt{abtahi@du.ac.ir}}


\begin{abstract}
  Let $\mathbb{B}(\mathcal{H})$ denote the $C^{\ast}$-algebra of all bounded
  linear operators on a Hilbert space $\big(\mathcal{H}, \langle\cdot, \cdot\rangle\big)$.
  Given a positive operator $A\in\B(\h)$, and a number $\lambda\in [0,1]$,
  a seminorm ${\|\cdot\|}_{(A,\lambda)}$ is defined on the set
  $\B_{A^{1/2}}(\h)$ of all operators in $\B(\h)$ having an $A^{1/2}$-adjoint.
  The seminorm ${\|\cdot\|}_{(A,\lambda)}$ is a combination of the sesquilinear form
  ${\langle \cdot, \cdot\rangle}_A$ and its induced seminorm ${\|\cdot\|}_A$.
  A characterization of Birkhoff--James orthogonality for operators with respect
  to the discussed seminorm is given. Moving $\lambda$ along the interval $[0,1]$, a wide spectrum
  of seminorms are obtained, having the $A$-numerical radius $w_A(\cdot)$ at the beginning (associated with $\lambda=0$)
  and the $A$-operator seminorm ${\|\cdot\|}_A$ at the end (associated with $\lambda=1$). Moreover, if $A=I$ the identity
  operator, the classical operator norm and numerical radius are obtained. Therefore, the results in this
  paper are significant extensions and generalizations of known results in this area.
\end{abstract}

\maketitle

\section{Introduction and preliminaries}

Let $\mathbb{B}(\mathcal{H})$ be the algebra of all bounded linear operators on a complex Hilbert space
$\mathcal{H}$ with an inner product $\langle \cdot,\cdot \rangle$ and the corresponding norm
$\|\cdot\| $. Let $I$ stand for the identity operator on $\mathcal{H}$.
Throughout this paper, we assume that $A\in\mathbb{B}(\mathcal{H})$ is a positive operator, which
induces a positive semidefinite sesquilinear
form ${\langle \cdot, \cdot\rangle}_A: \,\mathcal{H}\times \mathcal{H} \rightarrow \mathbb{C}$
defined by ${\langle x, y\rangle}_A = \langle Ax, y\rangle$.
We denote by ${\|\cdot\|}_A$ the seminorm induced by ${\langle \cdot, \cdot\rangle}_{A}$.
For the semi-Hilbertian space $\big(\mathcal{H}, {\|\cdot\|}_{A}\big)$ the $A$-Cauchy--Schwartz inequality holds, that is,
$\big|{\langle x, y\rangle}_{A}\big|\leq {\|x\|}_{A}{\|y\|}_{A}$ for all $x, y\in \mathcal{H}$.
For $T\in \mathbb{B}(\mathcal{H})$, an operator $S\in \mathbb{B}(\mathcal{H})$
is called an $A$-adjoint operator of $T$ if  ${\langle Tx, y\rangle}_{A} = {\langle x, Sy\rangle}_{A}$,
for every $x, y\in \mathcal{H}$.
The existence of an $A$-adjoint operator is not guaranteed.
The set of all operators admitting $A^{1/2}$-adjoints is denoted by $\mathbb{B}_{A^{1/2}}(\mathcal{H})$.
Clearly, ${\langle \cdot, \cdot\rangle}_{A}$ induces a seminorm on $\mathbb{B}_{A^{1/2}}(\mathcal{H})$.
Indeed, if $T\in\mathbb{B}_{A^{1/2}}(\mathcal{H})$, then
\begin{align*}
{\|T\|}_{A} =\sup\big\{{\|Tx\|}_{A}: \,\, x\in\mathcal{H}, {\|x\|}_A =1\big\} < +\infty.
\end{align*}

Notice that it may happen that ${\|T\|}_A = + \infty$
for some $T\in\mathbb{B}(\mathcal{H})\setminus \mathbb{B}_{A^{1/2}}(\mathcal{H})$.
It can be verified that, for $T\in\mathbb{B}_{A^{1/2}}(\mathcal{H})$,
we have ${\|Tx\|}_{A}\leq {\|T\|}_{A}{\|x\|}_{A}$ for all $x\in \mathcal{H}$.

Further, the $A$-minimum modulus of $T$, denoted by ${[T]}_{A}$, is defined as
\begin{align*}
{[T]}_{A} =\inf\big\{{\|Tx\|}_{A}: \,\, x\in\mathcal{H}, {\|x\|}_A =1\big\}.
\end{align*}

More details on semi-Hilbertian space operators can be found in \cite{Ar.Co.Go, Ma.Se.Su}.

The $A$-numerical radius and the $A$-Crawford number of $T\in\mathbb{B}(\mathcal{H})$ are defined,
respectively, by
\begin{align*}
w_{A}(T) & = \sup\Big\{\big|{\langle Tx, x\rangle}_A\big|: \,\,\, x\in \mathcal{H},\, {\|x\|}_A = 1\Big\}, \\
c_{A}(T) & = \inf\Big\{\big|{\langle Tx, x\rangle}_A\big|: \,\,\, x\in \mathcal{H},\, {\|x\|}_A = 1\Big\}.
\end{align*}

In particular, if we consider $A = I$ in the definitions of $A$-operator seminorm, $A$-minimum modulus,
$A$-numerical radius and $A$-Crawford number of $T$ then we get the classical operator norm,
minimum modulus, numerical radius and Crawford number, respectively, i.e.,
${\|T\|}_{A} = \|T\|, {[T]}_{A} = [T], w_{A}(T) = w(T)$ and $c_{A}(T) = c(T)$.
It is known that $w_{A}(\cdot)$ defines a seminorm on $\mathbb{B}_{A^{1/2}}(\mathcal{H})$, and that
\begin{equation*}
   \frac12 {\|T\|}_{A} \leq w_{A}(T)\leq {\|T\|}_{A},\ T\in \mathbb{B}_{A^{1/2}}(\mathcal{H}).
\end{equation*}

For other related information on
the numerical radius of operators in semi-Hilbertian spaces
we refer the reader to \cite{M.X.Z, Ro.Sa.Mi, Z.LAA} and the references therein.

In normed spaces, there are several notions of orthogonality, all of which are
generalizations of orthogonality in a Hilbert space. Among them, the Birkhoff--James orthogonality
is one of the most important. Given two elements $x,y$ in a normed space $(X,\|\cdot\|)$,
it is said that $x$ is orthogonal to $y$, in the Birkhoff--James sense \cite{B, J}, denoted by $x \perp_B y$, if
\begin{equation*}
  \|x + \xi y\| \geq \|x\|,\ \xi \in \C.
\end{equation*}

The Birkhoff--James orthogonality plays a central role in approximation theory.
On a Hilbert space $\mathcal{H}$, an operator $T\in\mathbb{B}(\mathcal{H})$ is
a best approximation of $S\in \mathbb{B}(\mathcal{H})$ from a linear subspace
$\mathbb{M}$ of $\mathbb{B}(\mathcal{H})$ if, and only if, $T$ is a Birkhoff--James
orthogonal projection of $S$ onto $\mathbb{M}$; see \cite{BCMWZ-2019} and references therein.
Bhatia and \v{S}emrl in \cite[Remark 3.1]{B.S} and Paul in \cite[Lemma 2]{Pa}
independently proved that $T\perp_B S$ if and only if there exists a sequence $\{x_n\}$
of unit vectors in $\mathcal{H}$ such that
\begin{align*}
  \lim_{n\to\infty} \|Tx_n\| = \|T\| \quad
\mbox{and} \quad \lim_{n\rightarrow\infty}\langle Tx_n, Sx_n\rangle = 0.
\end{align*}

It follows that if the Hilbert space $\mathcal{H}$ is finite-dimensional, then
$T\perp_B S$ if and only if there is a unit vector $x\in\mathcal{H}$ such that
$\|Tx\| = \|T\|$ and $\langle Tx, Sx\rangle = 0$. A number of authors have recently
extended the well-known result of Bhatia and \v{S}emrl;
see, e.g., \cite{B.F.P}, \cite{BCMWZ-2019}, \cite{Mal.Paul}, \cite{S.S.P}, \cite{Z.AFA}.

In this paper, we define a new seminorm on $\mathbb{B}_{A^{1/2}}(\mathcal{H})$, which generalizes
simultaneously the $A$-operator seminorm and the $A$-numerical radius.
We give a necessary and sufficient condition to hold that the seminorm of the sum
of elements in $\mathbb{B}_{A^{1/2}}(\mathcal{H})$ is equal to the sum of their seminorms.
We also characterize Birkhoff--James orthogonality of semi-Hilbertian space operators with respect to
this seminorm. Our results cover and extend some theorems in \cite{A.K, B.B, B.S, M.P.S, Z.AFA}.
In particular, related to a result due to Bhatia and \v{S}emrl \cite{B.S},
we give another equivalent condition of the Birkhoff--James orthogonality for Hilbert space operators.

\section{Main results}

We start the section with the following definition.

\begin{definition}\label{D.1}
  Let $(\h,\langle\cdot,\cdot\rangle)$ be a Hilbert space, $A\in\B(\h)$ be
  a positive operator and $\lambda\in [0,1]$. For every $T\in \mathbb{B}_{A^{1/2}}(\mathcal{H})$,
  define
  \begin{equation*}
    {\|T\|}_{_{(A, \lambda)}} = \sup \left\{\sqrt{\lambda{\|Tx\|}^2_{A} + (1 - \lambda){|{\langle Tx, x\rangle}_{A}|}^2}: \,x\in\mathcal{H}, {\|x\|}_{A} = 1\right\}.
  \end{equation*}
\end{definition}

\begin{remark}\label{R.2}
For $T\in \mathbb{B}_{A^{1/2}}(\mathcal{H})$ we have
${\|T\|}_{_{(A, 0)}} = w_{A}(T)$ and ${\|T\|}_{_{(A, 1)}} = {\|T\|}_{A}$.
\end{remark}

First of all, let us prove that ${\|\!\cdot\!\|}_{_{(A, \lambda)}}$ is a seminorm on $\mathbb{B}_{A^{1/2}}(\mathcal{H})$
sitting between
the $A$-numerical radius and $A$-operator seminorm.

\begin{proposition}\label{P.3}
  The function ${\|\!\cdot\!\|}_{_{(A, \lambda)}}$ is a seminorm on $\mathbb{B}_{A^{1/2}}(\mathcal{H})$
  and the following inequality holds for every $T\in \mathbb{B}_{A^{1/2}}(\mathcal{H})$;
  \begin{equation*}
  w_{A}(T) \leq {\|T\|}_{_{(A, \lambda)}} \leq {\|T\|}_{A}.
  \end{equation*}
\end{proposition}

\begin{proof}
Let $T, S\in \mathbb{B}_{A^{1/2}}(\mathcal{H})$.
It is trivial that ${\|\alpha T\|}_{_{(A, \lambda)}} = |\alpha| {\|T\|}_{_{(A, \lambda)}}$ for every $\alpha \in \mathbb{C}$.
Therefore, to show that ${\|\!\cdot\!\|}_{_{(A, \lambda)}}$ is a seminorm on $\mathbb{B}_{A^{1/2}}(\mathcal{H})$,
it suffices to show that ${\|\!\cdot\!\|}_{_{(A, \lambda)}}$ is subadditive.

Let $x\in\mathcal{H}$ with ${\|x\|}_{A} = 1$. We have
\begin{align*}
  \lambda \big\|(T & + S)x\big\|^2_{A} + (1 - \lambda){\big|{\langle (T + S)x, x\rangle}_{A}\big|}^2\\
  & = \lambda{\big\|Tx + Sx\big\|}^2_{A} + (1 - \lambda){\big|{\langle Tx, x\rangle}_{A} + {\langle Sx, x\rangle}_{A}\big|}^2\\
  & \leq \lambda\big({\|Tx\|}_{A} + {\|Sx\|}_{A}\big)^2
                  + (1-\lambda)\big(|{\langle Tx, x\rangle}_{A}| + |{\langle Sx, x\rangle}_{A}|\big)^2\\
  & \leq \big(\lambda{\|Tx\|}^2_{A}+(1-\lambda)|{\langle Tx, x\rangle}_{A}|^2\big) \\
  & \hspace{7em}+ 2\big(\lambda{\|Tx\|}_{A}\,{\|Sx\|}_{A}+(1-\lambda)|{\langle Tx, x\rangle}_{A}|\,|{\langle Sx, x\rangle}_{A}|\big)\\
  & \hspace{9em} + \big(\lambda{\|Sx\|}^2_{A}+(1-\lambda)|{\langle Sx, x\rangle}_{A}|^2\big)\\
  & \leq \big(\lambda{\|Tx\|}^2_{A}+(1-\lambda)|{\langle Tx, x\rangle}_{A}|^2\big)\\
  & \hspace{4em} +2\sqrt{\lambda{\|Tx\|}^2_{A}+(1-\lambda)|{\langle Tx, x\rangle}_{A}|^2}
  \,\sqrt{\lambda{\|Sx\|}^2_{A}+(1-\lambda)|{\langle Sx, x\rangle}_{A}|^2}\\
  & \hspace{8em} + \big(\lambda{\|Sx\|}^2_{A}+(1-\lambda)|{\langle Sx, x\rangle}_{A}|^2\big)\\
  & \hspace{12em} \big(\mbox{by the Cauchy--Bunyakovsky--Schwarz inequality}\big)\\
  & \leq {\|T\|}^2_{_{(A, \lambda)}} + 2 {\|T\|}_{_{(A, \lambda)}}\,{\|S\|}_{_{(A, \lambda)}} + {\|S\|}^2_{_{(A, \lambda)}}
  = \bigl({\|T\|}_{_{(A, \lambda)}} + {\|S\|}_{_{(A, \lambda)}}\bigr)^2.
\end{align*}

Therefore,
\begin{align*}
\sqrt{\lambda{\big\|(T + S)x\big\|}^2_{A} + (1 - \lambda){\big|{\langle (T + S)x, x\rangle}_{A}\big|}^2}
\leq {\|T\|}_{_{(A, \lambda)}} + {\|S\|}_{_{(A, \lambda)}}.
\end{align*}

Taking supremum over all $x\in\mathcal{H}$ with ${\|x\|}_{A} = 1$ in the above inequality,
we get
\begin{align*}
{\|T + S\|}_{_{(A, \lambda)}} \leq {\|T\|}_{_{(A, \lambda)}} + {\|S\|}_{_{(A, \lambda)}}.
\end{align*}

Further, by the $A$-Cauchy--Schwarz inequality, for every $A$-unit vector $x\in\mathcal{H}$ we have
\begin{align*}
|{\langle Tx, x\rangle}_{A}|^2 &= \lambda |{\langle Tx, x\rangle}_{A}|^2 + (1-\lambda)|{\langle Tx, x\rangle}_{A}|^2\\
& \leq  \lambda{\|Tx\|}^2_{A} + (1-\lambda)|{\langle Tx, x\rangle}_{A}|^2\\
& \leq \lambda{\|Tx\|}^2_{A} + (1-\lambda){\|Tx\|}^2_{A} = {\|Tx\|}^2_{A},
\end{align*}
and hence
\begin{align*}
|{\langle Tx, x\rangle}_{A}| \leq \sqrt{\lambda{\|Tx\|}^2_{A} + (1-\lambda)|{\langle Tx, x\rangle}_{A}|^2}
\leq {\|Tx\|}_{A}.
\end{align*}

This implies that $w_{A}(T) \leq {\|T\|}_{_{(A, \lambda)}} \leq {\|T\|}_{A}$
which completes the proof.
\end{proof}

\begin{remark}\label{R.4}

Another lower and upper bound for the seminorm ${\| \cdot \|}_{_{(A, \lambda)}}$
of bounded linear operators can be presented as follows.
Let $T\in \mathbb{B}_{A^{1/2}}(\mathcal{H})$.
For every $x\in\mathcal{H}$ with ${\|x\|}_{A} = 1$, by the arithmetic geometric mean inequality, we have
\begin{align*}
2\sqrt{\lambda(1-\lambda)}c_{A}(T){\|Tx\|}_{A} & \leq 2\sqrt{\lambda(1-\lambda)}|{\langle Tx, x\rangle}_{A}|{\|Tx\|}_{A}\\
& \leq \lambda{\|Tx\|}^2_{A} + (1-\lambda)|{\langle Tx, x\rangle}_{A}|^2 \leq {\|T\|}^2_{_{(A, \lambda)}}
\end{align*}

Hence
\begin{align*}
2\sqrt{\lambda(1-\lambda)}c_{A}(T){\|Tx\|}_{A} \leq {\|T\|}^2_{_{(A, \lambda)}}.
\end{align*}

Taking supremum over all $x\in\mathcal{H}$ with ${\|x\|}_{A} = 1$ in the above inequality, we arrive at
\begin{align*}
2\sqrt{\lambda(1-\lambda)}c_{A}(T){\|T\|}_{A} \leq {\|T\|}^2_{_{(A, \lambda)}}.
\end{align*}

On the other hand, for every $A$-unit vector $x\in\mathcal{H}$, by \cite[p. 172]{Z.LAA}, we have
\begin{align*}
{\|Tx\|}^2_{A} \leq 2w_{A}(T)\Bigl(w_{A}(T) + \sqrt{w^2_{A}(T) - c^2_{A}(T)}\,\Bigr).
\end{align*}

Therefore,
\begin{align*}
\lambda{\|Tx\|}^2_{A} &+ (1-\lambda)|{\langle Tx, x\rangle}_{A}|^2 \\
&\leq 2\lambda w_{A}(T)\Bigl(w_{A}(T) + \sqrt{w^2_{A}(T) - c^2_{A}(T)}\,\Bigr) + (1-\lambda)w^2_{A}(T)\\
& = (1+\lambda)w^2_{A}(T) + 2\lambda w_{A}(T)\sqrt{w^2_{A}(T) - c^2_{A}(T)},
\end{align*}
which yields
\begin{align*}
{\|T\|}^2_{_{(A, \lambda)}} \leq (1+\lambda)w^2_{A}(T) + 2\lambda w_{A}(T)\sqrt{w^2_{A}(T) - c^2_{A}(T)}.
\end{align*}

\end{remark}

In the following theorem, we give a necessary and sufficient condition for the
equality ${\|T+S\|}_{_{(A, \lambda)}} = {\|T\|}_{_{(A, \lambda)}} + {\|S\|}_{_{(A, \lambda)}}$
to hold in $\mathbb{B}_{A^{1/2}}(\mathcal{H})$.

\begin{theorem}\label{T.5}

Let $T, S\in \mathbb{B}_{A^{1/2}}(\mathcal{H})$.
The following statements are equivalent.
\begin{itemize}
\item[(i)] ${\|T+S\|}_{_{(A, \lambda)}} = {\|T\|}_{_{(A, \lambda)}} + {\|S\|}_{_{(A, \lambda)}}$.
\item[(ii)] There exists a sequence $\{x_n\}$ of $A$-unit vectors in $\mathcal{H}$ such that
\begin{align*}
\lim_{n\to\infty} \Big(\lambda {\langle Sx_n, Tx_n\rangle}_{A}
     + (1-\lambda){\langle x_n, Tx_n \rangle}_{A} {\langle Sx_n,x_n \rangle}_{A}\Big)
     = {\|T\|}_{_{(A, \lambda)}}\,{\|S\|}_{_{(A, \lambda)}}.
\end{align*}
\end{itemize}

\end{theorem}

\begin{proof}

(i)$\Rightarrow$(ii) Let ${\|T+S\|}_{_{(A, \lambda)}} = {\|T\|}_{_{(A, \lambda)}} + {\|S\|}_{_{(A, \lambda)}}$.
Then there exists a sequence of $A$-unit vectors $\{x_n\}$ in $\mathcal{H}$ such that
\begin{align}\label{I.1.T.5}
\lim_{n\to\infty} \Big(\lambda{\|(T+S)x_n\|}^2_{A} + (1 - \lambda){|{\langle (T+S)x_n, x_n\rangle}_{A}|}^2\Big)
= \big({\|T\|}_{_{(A, \lambda)}} + {\|S\|}_{_{(A, \lambda)}}\big)^2.
\end{align}

For every $n\in \mathbb{N}$, we have
\begin{align*}
\lambda&{\|(T+S)x_n\|}^2_{A} + (1 - \lambda){|{\langle (T+S)x_n, x_n\rangle}_{A}|}^2\\
& = \lambda{\|Tx_n\|}^2_{A} + 2\mathfrak{Re}\Big(\lambda{\langle Sx_n, Tx_n\rangle}_{A}\Big) + \lambda{\|Sx_n\|}^2_{A} + (1 - \lambda){|{\langle Tx_n, x_n\rangle}_{A}|}^2\\
& \qquad + 2\mathfrak{Re}\Big((1-\lambda){\langle x_n, Tx_n \rangle}_{A} {\langle Sx_n,x_n \rangle}_{A}\Big)
+ (1 - \lambda){|{\langle Sx_n, x_n\rangle}_{A}|}^2\\
& =  \lambda{\|Tx_n\|}^2_{A} + (1 - \lambda){|{\langle Tx_n, x_n\rangle}_{A}|}^2
+ \lambda{\|Sx_n\|}^2_{A} + (1 - \lambda){|{\langle Sx_n, x_n\rangle}_{A}|}^2\\
& \qquad + 2\mathfrak{Re}\Big(\lambda{\langle Sx_n, Tx_n\rangle}_{A} +
(1-\lambda){\langle x_n, Tx_n \rangle}_{A} {\langle Sx_n,x_n \rangle}_{A}\Big)\\
& \leq {\|T\|}^2_{_{(A, \lambda)}} + {\|S\|}^2_{_{(A, \lambda)}} + 2\Big|\lambda{\langle Sx_n, Tx_n\rangle}_{A} +
(1-\lambda){\langle x_n, Tx_n \rangle}_{A} {\langle Sx_n,x_n \rangle}_{A}\Big|\\
& \leq {\|T\|}^2_{_{(A, \lambda)}} + {\|S\|}^2_{_{(A, \lambda)}} + 2\Big(\lambda{\|Tx_n\|}_{A}\,{\|Sx_n\|}_{A} +
(1-\lambda)|{\langle Tx_n, x_n \rangle}_{A}|\,|{\langle Sx_n,x_n \rangle}_{A}|\Big)\\
& \qquad \qquad \qquad \qquad \qquad \qquad \qquad \qquad \big(\mbox{by the $A$-Cauchy--Schwarz inequality}\big)
\end{align*}
\begin{align*}
& \leq {\|T\|}^2_{_{(A, \lambda)}} + {\|S\|}^2_{_{(A, \lambda)}}\\
& \qquad + 2\sqrt{\lambda{\|Tx_n\|}^2_{A} + (1 - \lambda){|{\langle Tx_n, x_n\rangle}_{A}|}^2}\,
\sqrt{\lambda{\|Sx_n\|}^2_{A} + (1 - \lambda){|{\langle Sx_n, x_n\rangle}_{A}|}^2}\\
& \qquad \qquad \qquad \qquad \qquad \quad \big(\mbox{by the Cauchy--Bunyakovsky--Schwarz inequality}\big)\\
& \leq {\|T\|}^2_{_{(A, \lambda)}} + {\|S\|}^2_{_{(A, \lambda)}} + 2{\|T\|}_{_{(A, \lambda)}}\,{\|S\|}_{_{(A, \lambda)}}
= \big({\|T\|}_{_{(A, \lambda)}} + {\|S\|}_{_{(A, \lambda)}}\big)^2,
\end{align*}

and therefore, from \eqref{I.1.T.5}, we obtain
\begin{align}\label{I.2.T.5}
\lim_{n\to\infty}\mathfrak{Re}\Big(\lambda{\langle Sx_n, Tx_n\rangle}_{A} +
(1-\lambda){\langle x_n, Tx_n \rangle}_{A} {\langle Sx_n,x_n \rangle}_{A}\Big) = {\|T\|}_{_{(A, \lambda)}}\,{\|S\|}_{_{(A, \lambda)}}.
\end{align}

In addition, for every $n\in \mathbb{N}$, we have
\begin{align*}
\mathfrak{Re}^2&\Big(\lambda{\langle Sx_n, Tx_n\rangle}_{A} +
(1-\lambda){\langle x_n, Tx_n \rangle}_{A} {\langle Sx_n,x_n \rangle}_{A}\Big) \\
&+ \mathfrak{Im}^2\Big(\lambda{\langle Sx_n, Tx_n\rangle}_{A} +
(1-\lambda){\langle x_n, Tx_n \rangle}_{A} {\langle Sx_n,x_n \rangle}_{A}\Big)\\
&= \Big|\lambda{\langle Sx_n, Tx_n\rangle}_{A} +
(1-\lambda){\langle x_n, Tx_n \rangle}_{A} {\langle Sx_n,x_n \rangle}_{A}\Big|^2\leq {\|T\|}^2_{_{(A, \lambda)}}\,{\|S\|}^2_{_{(A, \lambda)}},
\end{align*}

and so by \eqref{I.2.T.5}, we conclude that
\begin{align*}
\lim_{n\to\infty}\mathfrak{Im}\Big(\lambda{\langle Sx_n, Tx_n\rangle}_{A} +
(1-\lambda){\langle x_n, Tx_n \rangle}_{A} {\langle Sx_n,x_n \rangle}_{A}\Big) = 0.
\end{align*}

It follows from \eqref{I.2.T.5} that
\begin{align*}
\lim_{n\to\infty}\Big(\lambda {\langle Sx_n, Tx_n\rangle}_{A}
+(1-\lambda){\langle x_n, Tx_n \rangle}_{A} {\langle Sx_n,x_n \rangle}_{A}\Big) = {\|T\|}_{_{(A, \lambda)}}\,{\|S\|}_{_{(A, \lambda)}}.
\end{align*}

(ii)$\Rightarrow$(i) Suppose that for a sequence of $A$-unit vectors $\{x_n\}$ in $\mathcal{H}$
we have
\begin{align*}
\lim_{n\to\infty}\Big(\lambda {\langle Sx_n, Tx_n\rangle}_{A}
+(1-\lambda){\langle x_n, Tx_n \rangle}_{A} {\langle Sx_n,x_n \rangle}_{A}\Big) = {\|T\|}_{_{(A, \lambda)}}\,{\|S\|}_{_{(A, \lambda)}}.
\end{align*}

Hence
\begin{align*}
\displaystyle{\lim_{n\to\infty}}\mathfrak{Re}\Big(\lambda {\langle Sx_n, Tx_n\rangle}_{A}
+(1-\lambda){\langle x_n, Tx_n \rangle}_{A} {\langle Sx_n,x_n \rangle}_{A}\Big) = {\|T\|}_{_{(A, \lambda)}}\,{\|S\|}_{_{(A, \lambda)}}.
\end{align*}

Since, for every $n\in \mathbb{N}$,
\begin{align*}
\Big|&\lambda{\langle Sx_n, Tx_n\rangle}_{A} +
(1-\lambda){\langle x_n, Tx_n \rangle}_{A} {\langle Sx_n,x_n \rangle}_{A}\Big|^2\\
& \leq \Big(\lambda{\|Tx_n\|}^2_{A} + (1 - \lambda){|{\langle Tx_n, x_n\rangle}_{A}|}^2\Big)\,
\Big(\lambda{\|Sx_n\|}^2_{A} + (1 - \lambda){|{\langle Sx_n, x_n\rangle}_{A}|}^2\Big)\\
& \leq \Big(\lambda{\|Tx_n\|}^2_{A} + (1 - \lambda){|{\langle Tx_n, x_n\rangle}_{A}|}^2\Big)\,{\|S\|}^2_{_{(A, \lambda)}}
\leq {\|T\|}^2_{_{(A, \lambda)}}\,{\|S\|}^2_{_{(A, \lambda)}},
\end{align*}

we obtain
\begin{align*}
\displaystyle{\lim_{n\to\infty}}
\Big(\lambda{\|Tx_n\|}^2_{A} + (1 - \lambda){|{\langle Tx_n, x_n\rangle}_{A}|}^2\Big) = {\|T\|}^2_{_{(A, \lambda)}}.
\end{align*}

By a similar argument, we get
\begin{align*}
\lim_{n\to\infty}
\Big(\lambda{\|Sx_n\|}^2_{A} + (1 - \lambda){|{\langle Sx_n, x_n\rangle}_{A}|}^2\Big) = {\|S\|}^2_{_{(A, \lambda)}}.
\end{align*}

Therefore,
\begin{align*}
\big({\|T\|}_{_{(A, \lambda)}} + {\|S\|}_{_{(A, \lambda)}}\big)^2 &=\displaystyle{\lim_{n\to\infty}}
\Big(\lambda{\|Tx_n\|}^2_{A} + (1 - \lambda){|{\langle Tx_n, x_n\rangle}_{A}|}^2\Big)\\
&\quad + 2\displaystyle{\lim_{n\to\infty}}\mathfrak{Re}\Big(\lambda {\langle Sx_n, Tx_n\rangle}_{A}
+(1-\lambda){\langle x_n, Tx_n \rangle}_{A} {\langle Sx_n,x_n \rangle}_{A}\Big)\\
&\qquad + \displaystyle{\lim_{n\to\infty}}
\Big(\lambda{\|Sx_n\|}^2_{A} + (1 - \lambda){|{\langle Sx_n, x_n\rangle}_{A}|}^2\Big)
\\& = \displaystyle{\lim_{n\to\infty}}\Big(\lambda{\|(T+S)x_n\|}^2_{A} + (1 - \lambda){|{\langle (T+S)x_n, x_n\rangle}_{A}|}^2\Big)\\
&\leq {\|T+S\|}^2_{_{(A, \lambda)}}
\leq \big({\|T\|}_{_{(A, \lambda)}} + {\|S\|}_{_{(A, \lambda)}}\big)^2.
\end{align*}

Hence ${\|T+S\|}_{_{(A, \lambda)}} = {\|T\|}_{_{(A, \lambda)}} + {\|S\|}_{_{(A, \lambda)}}$.
\end{proof}

As an immediate consequence of the preceding theorem,
we obtain the following result due to Barraa and Boumazgour \cite{B.B}.

\begin{corollary}[{\cite[Theorem~2.1]{B.B}}]\label{C.6}

Let $T, S\in \mathbb{B}(\mathcal{H})$. Then $\|T + S\| = \|T\| + \|S\|$ if and only if
there exists a sequence $\{x_n\}$ of unit vectors in $\mathcal{H}$ such that
\begin{align*}
\displaystyle{\lim_{n\to\infty}}\langle Sx_n, Tx_n\rangle = \|T\|\|S\|.
\end{align*}

\end{corollary}

\begin{proof}
  Let $A =I$ and $\lambda = 1$, and apply Theorem \ref{T.5}.
\end{proof}

As another application of Theorem \ref{T.5}, by letting $A = I$ and $\lambda = 0$,
we get a characterization of the equality $w(T + S) = w(T) + w(S)$ for Hilbert space operators (see \cite{A.K}).

\begin{corollary}[{\cite[Proposition~3.6]{A.K}}]\label{C.6.5}

For $T, S\in \mathbb{B}(\mathcal{H})$,
the equality $w(T + S) = w(T) + w(S)$ holds if and only if
there exists a sequence $\{x_n\}$ of unit vectors in $\mathcal{H}$ such that
\begin{align*}
\lim_{n\to\infty}\langle x_n, Tx_n \rangle \langle Sx_n,x_n \rangle = w(T)w(S).
\end{align*}

\end{corollary}

An operator $T\in \mathbb{B}(\mathcal{H})$ is called \emph{Birkhoff--James numerical radius orthogonal}
to $S\in \mathbb{B}(\mathcal{H})$, denoted by $T\perp^w_BS$, if $w(T+\xi S)\geq w(T)$,
for all $\xi \in \mathbb{C}$. See \cite{M.P.S} for characterization of the Birkhoff-James orthogonality
with respect to numerical radius for Hilbert space operators. Analogously, we introduce a concept of
$(A, \lambda)$-Birkhoff--James orthogonality for semi-Hilbertian space operators.

\begin{definition}\label{de.31}
  Let $T, S\in \mathbb{B}_{A^{1/2}}(\mathcal{H})$. The operator $T$ is called $(A, \lambda)$-Birkhoff--James
  orthogonal to $S$, in short, $T \perp_{_{(A, \lambda)}} S$, if
  \begin{equation*}
     \|T + \xi S\|_{_{(A, \lambda)}}\geq {\|T\|}_{_{(A, \lambda)}}, \ \xi \in \mathbb{C}.
  \end{equation*}
\end{definition}

Obviously, this is a generalization of both the concept of Birkhoff--James orthogonality
and the concept of Birkhoff--James numerical radius orthogonality of Hilbert space operators.

In the next theorem, some characterizations of $(A, \lambda)$-Birkhoff--James orthogonality for bounded
linear operators in semi-Hilbertian spaces are presented.

\begin{theorem}\label{T.7}

  Let $T, S\in \mathbb{B}_{A^{1/2}}(\mathcal{H})$.
  Then the following statements are equivalent.
  \begin{itemize}
  \item[(i)] For each $\theta\in[0, 2\pi)$ there exits a sequence $\{x_{n}\}$
   of $A$-unit vectors in $\mathcal{H}$ such that the following two conditions hold.
  \begin{itemize}
  \item[(i-1)]
  $\displaystyle{\lim_{n\to\infty}}\Big(\lambda{\|Tx_{n}\|}^2_{A} + (1-\lambda){|{\langle Tx_{n}, x_{n}\rangle}_{A}|}^2\Big)
  = {\|T\|}^2_{_{(A, \lambda)}}$,
  \item[(i-2)]
  $\displaystyle{\lim_{n\to\infty}}\mathfrak{Re}\Big(e^{i\theta}\lambda{\langle Sx_{n}, Tx_{n}\rangle}_{A}
  + e^{i\theta}(1-\lambda){\langle x_{n}, Tx_{n}\rangle}_{A} {\langle Sx_{n}, x_{n}\rangle}_{A}\Big)\geq 0$.
  \end{itemize}
  \item[(ii)] For all $\xi \in \mathbb{C}$,
  ${\|T + \xi S\|}^2_{_{(A, \lambda)}}\geq {\|T\|}^2_{_{(A, \lambda)}} + |\xi|^2m^2_{_{(A, \lambda)}}(S)$,
  where
  \begin{align*}
  m_{_{(A, \lambda)}}(S) = \inf \Big\{\sqrt{\lambda{\|Sx\|}^2_{A} + (1 - \lambda){|{\langle Sx, x\rangle}_{A}|}^2}: \,
    x\in\mathcal{H}, {\|x\|}_{A} = 1\Big\}.
  \end{align*}
  \item[(iii)] $T \perp_{_{(A, \lambda)}} S$.
  \end{itemize}

\end{theorem}

\begin{proof}

(i)$\Rightarrow$(ii) Suppose that (i) holds and let $\xi\in \C$.
Then there exits $\theta \in [0,2\pi)$ such that $\xi =|\xi|e^{i\theta}$.
Let $\{x_{n}\}$ be a sequence of $A$-unit vectors in $\mathcal{H}$ such that (i-1) and (i-2)
hold. For $n\in \mathbb{N}$ we have
\begin{align*}
{\|T + \xi S\|}^2_{_{(A, \lambda)}} & \geq \lambda{\big\|Tx_{n} +|\xi| e^{i\theta} Sx_{n}\big\|}^2_{A}
+ (1 - \lambda){\big|{\langle (T+|\xi| e^{i\theta} S)x_n, x_{n}\rangle}_{A}\big|}^2\\
& = \lambda{\|Tx_n\|}^2_{A} + (1 - \lambda){|{\langle Tx_n, x_n\rangle}_{A}|}^2\\
& \qquad + |\xi|^2\Big(\lambda{\|Sx_n\|}^2_{A} + (1 - \lambda){|{\langle Sx_n, x_n\rangle}_{A}|}^2\Big)\\
& \qquad \quad + 2|\xi|\mathfrak{Re}\Big(e^{i\theta}\lambda{\langle Sx_n, Tx_n\rangle}_{A} +
e^{i\theta}(1-\lambda){\langle x_n, Tx_n \rangle}_{A} {\langle Sx_n,x_n \rangle}_{A}\Big),
\end{align*}
and so
\begin{align*}
{\|T + \xi S\|}^2_{_{(A, \lambda)}} & \geq {\|T\|}^2_{_{(A, \lambda)}} +
|\xi|^2 \limsup_{n\to\infty}\Big(\lambda{\|Sx_n\|}^2_{A} + (1 - \lambda){|{\langle Sx_n, x_n\rangle}_{A}|}^2\Big)\\
& \geq {\|T\|}^2_{_{(A, \lambda)}} + |\xi|^2m^2_{_{(A, \lambda)}}(S).
\end{align*}

Hence ${\|T + \xi S\|}^2_{_{(A, \lambda)}}\geq {\|T\|}^2_{_{(A, \lambda)}} + |\xi|^2m^2_{_{(A, \lambda)}}(S)$.

(ii)$\Rightarrow$(iii) This implication is trivial.

(iii)$\Rightarrow$(i) Let $T \perp_{_{(A, \lambda)}} S$.
Then ${\|T + \xi S\|}_{_{(A, \lambda)}} \geq {\|T\|}_{_{(A, \lambda)}}$
for every $\xi \in \C$. We may assume that ${\|T\|}_{_{(A, \lambda)}}\neq0$ otherwise (i) trivially holds.
Let $\theta\in [0,2\pi)$.
Thus ${\|T\|}_{_{(A, \lambda)}} \leq {\big\|T + \frac{e^{i\theta}}{n} S\big\|}_{_{(A, \lambda)}}$
for all $n\in \mathbb{N}$.
Since ${\|T\|}_{_{(A, \lambda)}}>0$, for sufficiently large $n$, we have
\begin{align*}
0 < {\|T\|}_{_{(A, \lambda)}} - \frac{1}{n^2} < {\|T\|}_{_{(A, \lambda)}}
\leq {\Big\|T + \frac{e^{i\theta}}{n}S\Big\|}_{_{(A, \lambda)}}.
\end{align*}

So, there exits a sequence $\{x_{n}\}$ of $A$-unit vectors in $\mathcal{H}$ such that
\begin{align}\label{I.1.T.7}
\Bigl({\|T\|}_{_{(A, \lambda)}} - \frac{1}{n^2}\Bigr)^2
  < \lambda \Big\|(T+\frac{e^{i\theta}}{n}S)x_{n}\Big\|^2_{A}
+ (1 - \lambda) \Big|{\langle (T+\frac{e^{i\theta}}{n}S)x_n, x_{n}\rangle}_{A}\Big|^2.
\end{align}

It follows from \eqref{I.1.T.7} that
\begin{align*}
{\|T\|}^2_{_{(A, \lambda)}} &- \frac{2}{n^2}{\|T\|}_{_{(A, \lambda)}} + \frac{1}{n^4}\\
& < \lambda{\|Tx_n\|}^2_{A} + (1 - \lambda){|{\langle Tx_n, x_n\rangle}_{A}|}^2 \\
& \quad + \frac{1}{n^2}\Big(\lambda{\|Sx_n\|}^2_{A} + (1 - \lambda){|{\langle Sx_n, x_n\rangle}_{A}|}^2\Big)\\
& \quad + \frac{2}{n}\mathfrak{Re}\Big(e^{i\theta}\lambda{\langle Sx_n, Tx_n\rangle}_{A} +
e^{i\theta}(1-\lambda){\langle x_n, Tx_n \rangle}_{A} {\langle Sx_n,x_n \rangle}_{A}\Big),
\end{align*}
and hence
\begin{align*}
\frac{n}{2}\Big({\|T\|}^2_{_{(A, \lambda)}} &- \lambda{\|Tx_n\|}^2_{A} - (1 - \lambda){|{\langle Tx_n, x_n\rangle}_{A}|}^2\Big)\\
& < \frac{1}{n}{\|T\|}_{_{(A, \lambda)}} - \frac{1}{2n^3} + \frac{1}{2n}\Big(\lambda{\|Sx_n\|}^2_{A} + (1 - \lambda){|{\langle Sx_n, x_n\rangle}_{A}|}^2\Big)\\
& \quad + \mathfrak{Re}\Big(e^{i\theta}\lambda{\langle Sx_n, Tx_n\rangle}_{A} +
e^{i\theta}(1-\lambda){\langle x_n, Tx_n \rangle}_{A} {\langle Sx_n,x_n \rangle}_{A}\Big).
\end{align*}

Since ${\|T\|}^2_{_{(A, \lambda)}} - \lambda{\|Tx_n\|}^2_{A} - (1 - \lambda){|{\langle Tx_n, x_n\rangle}_{A}|}^2\geq 0$, we obtain
\begin{align}\label{I.2.T.7}
0 &< \frac{1}{n}{\|T\|}_{_{(A, \lambda)}} - \frac{1}{2n^3}
+ \frac{1}{2n}\Big(\lambda{\|Sx_n\|}^2_{A} + (1 - \lambda){|{\langle Sx_n, x_n\rangle}_{A}|}^2\Big)\nonumber\\
& \qquad \quad + \mathfrak{Re}\Big(e^{i\theta}\lambda{\langle Sx_n, Tx_n\rangle}_{A} +
e^{i\theta}(1-\lambda){\langle x_n, Tx_n \rangle}_{A} {\langle Sx_n,x_n \rangle}_{A}\Big).
\end{align}

By letting $n\to\infty$ in \eqref{I.2.T.7} and passing through a subsequence if necessary, we get
\begin{align*}
\displaystyle{\lim_{n\to\infty}}\mathfrak{Re}\Big(e^{i\theta}\lambda{\langle Sx_n, Tx_n\rangle}_{A} +
e^{i\theta}(1-\lambda){\langle x_n, Tx_n \rangle}_{A} {\langle Sx_n,x_n \rangle}_{A}\Big)\geq 0.
\end{align*}

Further, by \eqref{I.1.T.7}, we have
\begin{align*}
\lambda{\|Tx_n\|}^2_{A} &+ (1 - \lambda){|{\langle Tx_n, x_n\rangle}_{A}|}^2\\
& > \Big({\|T\|}_{_{(A, \lambda)}} - \frac{1}{n^2}\Big)^2
   - \frac{1}{n^2}\Big(\lambda{\|Sx_n\|}^2_{A} + (1 - \lambda){|{\langle Sx_n, x_n\rangle}_{A}|}^2\Big)\\
& \quad - \frac{2}{n}\mathfrak{Re}\Big(e^{i\theta}\lambda{\langle Sx_n, Tx_n\rangle}_{A} +
e^{i\theta}(1-\lambda){\langle x_n, Tx_n \rangle}_{A} {\langle Sx_n,x_n \rangle}_{A}\Big),
\end{align*}
and therefore, by letting $n\to\infty$, we obtain
\begin{align*}
\displaystyle{\lim_{n\to\infty}}\Big(\lambda{\|Tx_{n}\|}^2_{A} + (1-\lambda){|{\langle Tx_{n}, x_{n}\rangle}_{A}|}^2\Big)
\geq {\|T\|}^2_{_{(A, \lambda)}}.
\end{align*}

Since $\displaystyle{\lim_{n\to\infty}}\Big(\lambda{\|Tx_{n}\|}^2_{A} + (1-\lambda){|{\langle Tx_{n}, x_{n}\rangle}_{A}|}^2\Big)
\leq {\|T\|}^2_{_{(A, \lambda)}}$, we conclude that
\begin{align*}
\displaystyle{\lim_{n\to\infty}}\Big(\lambda{\|Tx_{n}\|}^2_{A} + (1-\lambda){|{\langle Tx_{n}, x_{n}\rangle}_{A}|}^2\Big)
= {\|T\|}^2_{_{(A, \lambda)}}.
\end{align*}

Hence, (i-1) and (i-2) hold.

\end{proof}

The following corollary is a direct consequence of Theorem \ref{T.7}. It gives
a characterization of the Birkhoff--James orthogonality for Hilbert space operators.

\begin{corollary}\label{C.8}

Let $T, S\in \mathbb{B}(\mathcal{H})$. The following statements are equivalent.
\begin{itemize}
\item[(i)] For each $\theta\in[0, 2\pi)$ there exits a sequence $\{x_n\}$ of unit vectors in $\mathcal{H}$ such that
$\displaystyle{\lim_{n\to\infty}}\|Tx_n\| = \|T\|$
and
$\displaystyle{\lim_{n\to\infty}}\mathfrak{Re}\big(e^{i\theta}\langle Sx_n, Tx_n\rangle\big)\geq 0$.
\item[(ii)] For all $\xi \in \mathbb{C}$, $\|T + \xi S\|^2 \geq \|T\|^2 + |\xi|^2[S]^2$.
\item[(iii)] $T \perp_B S$.
\end{itemize}

\end{corollary}

\begin{proof}

Apply Theorem \ref{T.7} with $A=I$ and $\lambda = 1$.

\end{proof}

Finally, we get the following result due to Mal et al. \cite{M.P.S}.

\begin{corollary}[{\cite[Theorem~2.3]{M.P.S}}]\label{C.9}

Let $T, S\in \mathbb{B}(\mathcal{H})$. The following statements are equivalent.
\begin{itemize}
\item[(i)] For each $\theta\in[0, 2\pi)$ there exits a sequence $\{x_n\}$ of unit vectors in $\mathcal{H}$ such that
$\displaystyle{\lim_{n\to\infty}}|\langle Tx_n, x_n\rangle| = w(T)$
and
$\displaystyle{\lim_{n\to\infty}}\mathfrak{Re}\big(e^{i\theta}\langle x_n, Tx_n\rangle \langle Sx_n, x_n\rangle\big)\geq 0$.
\item[(ii)] For all $\xi \in \mathbb{C}$, $w^2(T + \xi S) \geq w^2(T) + |\xi|^2c^2(S)$.
\item[(iii)] $T \perp^{w}_B S$.
\end{itemize}

\end{corollary}

\begin{proof}

It follows immediately from Theorem \ref{T.7} with $A=I$ and $\lambda = 0$.
\end{proof}

\textbf{Acknowledgement.}
The authors thank the referees for helpful comments and suggestions.

\bibliographystyle{amsplain}

\end{document}